\documentclass[11pt]{35}
\usepackage{amsthm, amsmath, amssymb, amsbsy, amsfonts, latexsym, euscript}
\usepackage{calrsfs} 
\usepackage{graphicx}

\DeclareMathAlphabet{\varmathbb}{U}{pxsyb}{m}{n}

\def\leq{\leqslant}

\def\geq{\geqslant}
\def\phi{\varphi}

\def\kappa{\varkappa}

\newcommand{\D}{\mathrm{d}\kern0.2pt}%

\newcommand{\E}{\mathrm{e}\kern0.2pt} 
\newcommand{\ii}{\kern0.05em\mathrm{i}\kern0.05em}

\newcommand{\RR}{\mathbb{R}}%

\newtheorem{theorem}{\bf \indent Theorem}[section]
\newtheorem{proposition}{\bf \indent Proposition}[section]

\theoremstyle{remark}

\numberwithin{equation}{section}

\begin{document}

\noindent {\Large \bf Asymptotic mean value properties of meta- \\[3pt] and
panharmonic functions}

\vskip5mm

{\bf N. Kuznetsov}

\vskip-2pt {\small Laboratory for Mathematical Modelling of Wave Phenomena}
\vskip-4pt {\small Institute for Problems in Mechanical Engineering} \vskip-4pt
{\small Russian Academy of Sciences} \vskip-4pt {\small V.O., Bol'shoy pr. 61, St.
Petersburg, 199178} \vskip-4pt {\small Russian Federation} \vskip-4pt {\small
nikolay.g.kuznetsov@gmail.com}

\vskip4mm

\parbox{134mm} {\noindent Asymptotic mean value properties, their converse and some
other related results are considered for solutions to the $m$-dimensional Helmholtz
equation (metaharmonic functions) and solutions to its modified counterpart
(panharmonic functions). Some of these properties have no analogues for harmonic
functions.}

\vskip4pt

{\centering \section{Introduction and main theorem} }

\noindent Mean value properties of harmonic functions (solutions of the Laplace
equation; see \cite{ABR}, p.~25, about the origin of the term `harmonic') are well
known as well as various versions of assertions converse to these properties; see
the survey articles \cite{NV} and \cite{Ku}. On the other hand, analogous properties
of solutions to several other simple partial differential equations are studied less
thoroughly and, it may be said, fragmentary. Only recently the mean value property
for balls was obtained (see \cite{Ku1}) for solutions to the $m$-dimensional
Helmholtz equation:
\begin{equation}
\nabla^2 u + \lambda^2 u = 0 , \quad \lambda \in \RR \setminus \{0\} .
\label{Hh}
\end{equation}
Here and below, $\nabla = (\partial_1, \dots , \partial_m)$ denotes the gradient
operator in $\RR^m$ and $\partial_i = \partial / \partial x_i$.

In what follows, the term `metaharmonic function' is used as an abbreviation to
`solution of the Helmholtz equation'; it was I.~N. Vekua, who introduced this term
in 1943 in his still widely cited article \cite{Ve1} (its English translation can be
found in the book \cite{Ve2}). Also, the term `panharmonic function' (it was
proposed by Duffin \cite{D} in 1971) is used as a convenient equivalent to `solution
of the modified Helmholtz equation':
\begin{equation}
\nabla^2 u - \mu^2 u = 0 , \quad \mu \in \RR \setminus \{0\} . \label{MHh}
\end{equation}
The amended coefficient is used to distinguish this equation from \eqref{Hh}.
Indeed, these equations arise in different areas of mathematical physics: \eqref{Hh}
is usually considered as the reduced wave equation in which $\lambda$ is a wave
number, whereas the physical meaning of $\mu$ in \eqref{MHh} is completely
different. The reason is that its most important application is in the theory of
nuclear forces developed by Yukawa \cite{Y}; see the cited paper \cite{D}, where
this equation is referred to as the Yukawa equation.

In this note, the study of mean value properties of meta- and panharmonic functions
initiated in the author's papers \cite{Ku1} and \cite{Ku2} is continued. The aim is
to extend to these functions (for simplicity they are assumed to be real) a result
which for harmonic functions dates back to the classical theorems of Blaschke
\cite{Bla}, Priwaloff \cite{Pri} and Zaremba \cite{Za} (see also a discussion in
\cite{NV}, Sect.~9).

Before giving its precise formulation, let us introduce some notation used below.
For a point $x = (x_1, \dots, x_m) \in \RR^m$, $m \geq 2$, we denote by $B_r (x) =
\{ y : |y-x| < r \}$ the open ball of radius $r$ centred at $x$; the latter is
called admissible with respect to a domain $D \subset \RR^m$ provided $\overline{B_r
(x)} \subset D$, and $\partial B_r (x)$ is called admissible sphere in this case. If
$D$ has finite Lebesgue measure and a function $f$ is integrable over $D$, then
\[ M^\bullet (f, D) = \frac{1}{|D|} \int_{D} f (x) \, \D x
\]
is its volume mean value over $D$. Here and below $|D|$ is the domain's volume (area
if $D \subset \RR^2$), and $|B_r| = \omega_m r^m$ is the volume of $B_r$, where
$\omega_m = 2 \, \pi^{m/2} / [m \Gamma (m/2)]$ is the volume of unit ball; as usual
$\Gamma$ denotes the Gamma function. If $u \in C^0 (D)$, the mean value over an
admissible sphere $\partial B_r (x) \subset D$ is
\[ M^\circ (f, \partial B_r (x)) = \frac{1}{|\partial B_r|} \int_{\partial B_r (x)}
f (y) \, \D S_y \, ,
\]
where $|\partial B_r| = m \, \omega_m r^{m-1}$ and $\D S$ is the surface area
measure, is another useful notion.

\begin{theorem}[Blaschke, Priwaloff, Zaremba]
Let $D$ be a domain in $\RR^m$, $m \geq 2$, and let $u \in C^0 (D)$. Then $u$ is
harmonic in $D$ if and only if
\[\lim_{r \to +0} r^{-2} [ M^\bullet (u, B_r (x)) - u (x) ] = 0
\]
for every $x \in D$. The assertion also holds with $M^\bullet (u, B_r (x))$ changed
to $ M^\circ (u, \partial B_r (x))$.
\end{theorem}

Now, we are in a position to formulate the main result.

\begin{theorem}
Let $D$ be a domain in $\RR^m$, $m \geq 2$, and let $u \in C^2 (D)$. Then $u$ is
panharmonic in $D$ if and only if
\begin{equation}
\lim_{r \to +0} \frac{M^\bullet (u, B_r (x)) - u (x)}{r^2} = \frac{\mu^2 u (x)} {2
(m+2)} \quad \mbox{for every} \ x \in D . \label{par}
\end{equation}
The assertion also holds with $M^\bullet (u, B_r (x))$ changed to $M^\circ (u,
\partial B_r (x))$ in \eqref{par}, provided the right-hand side term is $\mu^2 u (x)
/ (2 m)$.

By changing $\mu^2 u (x)$ to $- \lambda^2 u (x)$, this assertion turns into a
necessary and sufficient condition of metaharmonicity for $u$.
\end{theorem}

\section{Proof of Theorem 1.2 and discussion}

The elementary proof of Theorem 1.2 given below is based on the well-known
relationship between the Laplacian and asymptotic mean values; see \cite{Br}, Ch.~2,
Sect.~2. It follows from Taylor's formula that
\[ u (x + y) - u (x) = y \cdot \nabla u (x) + 2^{-1} y \cdot [H_u (x)] y + o (r^2) \, ,
\]
which is valid for $u \in C^2 (D)$ at $x \in D$ as $r \to 0$ provided $B_r (x)$ is
admissible and $|y| \leq r$; here $H_u (x)$ denotes the Hessian matrix of $u$ at $x$
and ``$\cdot$'' stands for the inner product in $\RR^m$. Averaging each term of the
equality with respect to $y \in B_r (0)$, one obtains
\[ M^\bullet (u, B_r (x)) - u (x) = \frac{1}{2 |B_r|} \int_{B_r (0)} y \cdot [H_u (x)] 
y \, \D y + o (r^2) \, ,
\]
because the mean value of the first order term vanishes. It is straightforward to
calculate that
\begin{equation}
\lim_{r \to +0} \frac{M^\bullet (u, B_r (x)) - u (x)}{r^2} = \frac{\nabla^2 u (x)}
{2 (m+2)} \, , \quad x \in D . \label{vol}
\end{equation}
Similarly, it follows that
\begin{equation}
\lim_{r \to +0} \frac{M^\circ (u, \partial B_r (x)) - u (x)}{r^2} = \frac{\nabla^2 u
(x)} {2 m} \, , \quad x \in D , \label{sph}
\end{equation}
by virtue of averaging with respect to $y \in \partial B_r (0)$.

\begin{proof}[Proof of Theorem 1.2]
Let equality \eqref{par} hold, then combining it and formula \eqref{vol} one obtains
that $u$ is panharmonic in $D$; moreover, \eqref{sph} yields the analogous assertion
when $M^\circ (u, \partial B_r (x))$ stands in \eqref{par} instead of $M^\bullet (u,
\partial B_r (x))$.

Now, let $u$ be panharmonic in $D$, and so the mean value equality 
\[ M^\bullet (u, B_r (x)) = a^\bullet (\mu r) u (x) , \quad a^\bullet (\mu r) = 
\Gamma \left( \frac{m}{2} + 1 \right) \frac{I_{m/2} (\mu r)}{(\mu r / 2)^{m/2}} \, ,
\]
holds for every admissible $B_r (x) \subset D$ (see \cite{Ku1}, p.~95); here $I_\nu$
denotes the modified Bessel function of order $\nu$. Thus, \eqref{par} is true
provided
\[ \lim_{r \to +0} \frac{a^\bullet (\mu r) - 1}{(\mu r)^2} = \frac{1}{2 (m+2)} \, ,
\]
which follows from the definition of $I_{m/2}$.

In order to prove the necessity of
\begin{equation}
\lim_{r \to +0} \frac{M^\circ (u, \partial B_r (x)) - u (x)}{r^2} = \frac{\mu^2 u
(x)} {2 m} \quad \mbox{for every} \ x \in D  \label{par'}
\end{equation}
for panharmonicity of $u$, one has to apply the mean value equality 
\[ M^\circ (u, \partial B_r (x)) = a^\circ (\mu r) u (x) , \quad a^\circ (\mu r) = 
\Gamma \left( \frac{m}{2} \right) \frac{I_{(m-2)/2} (\mu r)}{(\mu r / 2)^{(m-2)/2}}
\, ,
\]
which holds for a panharmonic $u$ provided $B_r (x) \subset D$ is admissible; see
\cite{Ku1}, p.~94. Then \eqref{par} follows from
\[ \lim_{r \to +0} \frac{a^\circ (\mu r) - 1}{(\mu r)^2} = \frac{1}{2 m} \, ,
\]
which is true by the definition of $I_{(m-2)/2}$.

Considerations aimed at proving the analogous necessary and sufficient condition of
metaharmonicity are similar, but the coefficients $a^\bullet (\lambda r)$ and
$a^\circ (\lambda r)$ in the mean value equalities involve the Bessel function
$J_\nu$ instead of $I_\nu$; its order is the same as above in both cases.
\end{proof}

Let us consider some equalities related to \eqref{par} and \eqref{par'}. Their
immediate consequence is the following one
\begin{equation*}
(m+2) \lim_{r \to +0} \frac{M^\bullet (u, B_r (x)) - u (x)}{r^2} = m \lim_{r \to +0}
\frac{M^\circ (u, \partial B_r (x)) - u (x)}{r^2} \, , 
\end{equation*}
which is valid for every $x \in D$ provided $u$ is meta- or panharmonic. This is a
rare mean value property shared without any distinction by the both classes of
functions.

If $u$ is harmonic in $D$, then the analogous equality immediately follows from
Theorem~1.1:
\begin{equation}
\lim_{r \to +0} \frac{M^\bullet (u, B_r (x)) - M^\circ (u, \partial B_r (x))}{r^2} =
0 \quad \mbox{for every} \ x \in D . \label{har}
\end{equation}
To the best author's knowledge, the question whether \eqref{har} implies that $u$ is
harmonic was not investigated yet, and so the following assertion complements
Theorem~1.1.

\begin{proposition}
Let $D$ be a domain in $\RR^m$, $m \geq 2$, and let $u \in C^2 (D)$. Then $u$ is
harmonic in $D$ provided equality \eqref{har} holds for every $x \in D$.
\end{proposition}

\begin{proof}
Let us write \eqref{har} as follows:
\begin{equation*}
\lim_{r \to +0} \frac{M^\bullet (u, B_r (x)) - u (x)}{r^2} = \lim_{r \to +0}
\frac{M^\circ (u, \partial B_r (x)) - u (x)}{r^2} \quad \mbox{for every} \ x \in D.
\end{equation*}
Then, according to \eqref{vol} and \eqref{sph}, we have that
\[ \frac{\nabla^2 u (x)}{m+2} = \frac{\nabla^2 u (x)} {m} \quad \mbox{for every} \ 
x \in D ,
\]
which yields that $u$ is harmonic in $D$.
\end{proof}

Now, let us turn to properties related to \eqref{par}, but having no analogues for
harmonic functions. Since the equalities
\[ u (x) = M^\bullet (u, B_r (x)) / a^\bullet (\mu r) = M^\circ (u, \partial B_r (x)) 
/ a^\circ (\mu r) 
\]
hold for every $\ x \in D$ and every admissible $B_r (x)$ provided $u$ is
panharmonic, we see that \eqref{par} implies
\begin{eqnarray*}
&& \lim_{r \to +0} \frac{M^\bullet (u, B_r (x)) - u (x)}{r^2} = \frac{\mu^2
M^\bullet (u, B_r (x))} {2 (m+2) \, a^\bullet (\mu r)} \\ && \lim_{r \to +0}
\frac{M^\circ (u, \partial B_r (x)) - u (x)}{r^2} = \frac{\mu^2  M^\circ (u,
\partial B_r (x))} {2 \, m \, a^\circ (\mu r)}
\end{eqnarray*}
for these $x$ and $r$.

On the other hand, if each of these equalities holds for every $\ x \in D$ and every
admissible $B_r (x)$, then $u \in C^2 (D)$ is panharmonic in $D$. Indeed,
\eqref{vol} yields that the left-hand side of the first of these equalities is equal
to $\frac{\nabla^2 u (x)} {2 (m+2)}$, whereas letting $r \to 0$ on the right-hand
side one obtains $\frac{\mu^2 u (x)} {2 (m+2)}$, which implies that $u$ is
panharmonic in $D$. The same follows by combining the second equality and
\eqref{sph}.

For metaharmonic functions, the two equalities analogous to the last ones have $-
\lambda^2$ instead of $\mu^2$ and other functions $a^\bullet$ and $a^\circ$, but
again these equalities imply that $u$ is metaharmonic in~$D$.

\vspace{-12mm}

\renewcommand{\refname}{
\begin{center}{\Large\bf References}
\end{center}}
\makeatletter
\renewcommand{\@biblabel}[1]{#1.\hfill}
\makeatother

\end{document}